\documentclass{basm}

\usepackage{amssymb}
\usepackage{graphicx}
 \author{Julius Fergy T. Rabago}

\title{Forbidden Set of the Rational Difference Equation $x_{n+1} = x_n x_{n-k}/(ax_{n-k+1} +x_n x_{n-k+1} x_{n-k})$}

\RunningHead{Julius Fergy T. Rabago}%
            {Forbidden Set of $x_{n+1} = x_n x_{n-k}/(ax_{n-k+1} +x_n x_{n-k+1} x_{n-k})$}

\Keywords{Forbidden set, closed form solution, difference equation, open problem}

\MSC{39A10}

\Abstract{This short note aims to answer one of the open problems raised by F. Balibrea and A. Cascales in \cite{bc}. 
In particular, the forbidden set of the nonlinear difference equation $x_{n+1} = x_n x_{n-k}/(ax_{n-k+1} +x_n x_{n-k+1} x_{n-k})$, 
where $k$ is a positive integer and $a$ is a positive constant, is found by first computing the closed form solution of the given equation.
Additional results regarding the limiting properties and periodicity of its solutions are also discussed.
Numerical examples are also provided to illustrate the exhibited results.
Lastly, a possible generalization of this present work is offered as an open problem.}

\CopyRight {Julius Fergy T. Rabago, 2017}

\Address
{{\sc Julius Fergy T. Rabago}\\
{Department of Mathematics and Computer Science
College of Science, University of the Philippines Baguio
Governor Pack Road, Baguio City 2600, PHILIPPINES}\\
E-mail: \emph{jfrabago@gmail.com}
\medskip

}
 \Received { \ April 20, 2016}

\begin{document}

\maketitle

\section{Introduction}

Recently, various types of difference equations have been considered and examined (see, e.g., \cite{bc} and the papers cited therein).
These types of equations are of great importance in various fields of mathematics and areas of pure and applied sciences. 
In fact, they frequently appear as discrete mathematical models of many biological and environmental phenomena, such as population growth and predator-prey interactions \cite{ladas, kulenovic, kulenovic2}. 
They are also extensively used in deterministic formulations of dynamical phenomena in economics and social sciences \cite{sedaghat}.
One of the reasons why these equations are being studied is because they posses rich and complex dynamics (see, e.g., \cite{kulenovic}).
Some researchers, however, focus on the problem of finding closed form solutions of some solvable systems of nonlinear difference equations.
As a matter of fact, this line of research has become a growing interest in recent literature (see, e.g., \cite{elsayed, elsayed1, bacani2, rabago1, rabago2, rabago3, tollu, touafek, yazlik}, as well as the references therein).
In this work, we are also interested in finding a closed form solution of a certain class of difference equations, but, only as a way to solve a related problem.
To be more precise, we are interested in addressing the solution to one of the open problems posted by Balibrea and Cascales in \cite[Open Problem 3, Eq. 17]{bc} concerning the \emph{forbidden set} of a certain class of rational difference equations. 
Specifically, given fixed constants $k\in \mathbb{N}$ and $a>0$, we would like to find the forbidden set of the rational difference equation
\begin{equation}
\label{problem}
x_{n+1} = \dfrac{x_n x_{n-k}}{ax_{n-k+1} +x_n x_{n-k+1} x_{n-k}},
\end{equation}
with real initial conditions $\{x_n\}_{n=-k}^0$.
We shall determine the forbidden set of equation \eqref{problem} by first providing its closed form solution.

\begin{defn}
Given a rational difference equation $x_{n+1} = \frac{P(x_n,\ldots,x_{n-k})}{Q(x_n,\ldots,x_{n-k}}$ of order $k + 1$, 
where $P$ and $Q$ are two polynomials, there exists a subset $\mathcal{F}$ of $\mathbb{R}^{k+1}$ such that every initial condition of \eqref{problem} lying on 
$\mathcal{F}$ generates a finite solution $x_{-k},\ldots,x_m$ for which is impossible to construct $x_{m+1}$ because its $k + 1$ latest terms form a root of polynomial $Q$.
The set $\mathcal{F}$ is known as the \emph{forbidden set} of the equation.
\end{defn}

Consequently, the forbidden set $\mathcal{F}$ of a rational difference equation is the set of initial conditions which eventually map to a singularity,
or more intuitively, $\mathcal{F}$ is the set of initial conditions for which after a finite number of iterates we reach a value outside the domain of definition of the iteration function \cite{bc}. For some papers related to this topic, see \cite{bc} and \cite{palladino}, and the references cited therein.

Our main result, which answers the open problem \cite[Open Problem 3, Eq. 17]{bc}, is stated in Corollary \ref{cor1}. 
The term solution and iterate shall be used interchangeably throughout the rest of the paper. 
\section{Closed Form Solution and Forbidden Set of Equation \eqref{problem}}\label{sec2}

In this section, we derive the closed form solution of \eqref{problem}, and then deduce from the computed formula the forbidden set of the given equation.
To begin with, we provide some preliminary observations regarding the right side of equation \eqref{problem}. 
First, notice that the equation can be written as
\[
x_{n+1} = \dfrac{x_n x_{n-k}}{x_{n-k+1}(a + x_n x_{n-k})}.
\]
Clearly, this form suggests that the quantities $x_n x_{n-k} + a$ and $x_{n-k}$ should not be both zero, for all $n\in\mathbb{N}_0$, 
so that the sequence of iterates $\{x_n\}_{n=1}^{\infty}$ is well-define.
Hence, we assumed that $x_n x_{n-k} \neq a$ and $x_{n-k} \neq 0$, for all $n\in\mathbb{N}_0$.
These conditions shall be refined later on in the discussion by expressing them in terms of just the initial conditions $\{x_n\}_{n=-k}^0$ and the parameter $a$.
Meanwhile, if $x_n x_{n-k} = 1 - a$, for all $n\in \mathbb{N}_0$, and $a \neq 1$, then equation \eqref{problem} reduces to 
\begin{equation}
\label{simple}
x_n = \frac{1-a}{x_{n-k}}, \quad n \in \mathbb{N}_0.
\end{equation}
This implies that the solution sequence $\{x_n\}_{n=1}^{\infty}$ to \eqref{problem} is \emph{periodic} (see Definition \ref{periodicity}).
Indeed, substituting $x_{n-k} = (1-a)/x_{n-2k}$ in equation \eqref{simple} yields the equation (after an adjustment in the index) $x_{n+2k} = x_n$ ($n \in \mathbb{N}_0$).
Clearly, this equation shows that the sequence of iterates $\{x_n\}_{n=1}^{\infty}$ is periodic with period $2k$.

Now, in the sequel, we shall assume that $x_n x_{n-k} \neq 1 - a$ for all $n\in \mathbb{N}_0$ and $a\neq 1$.
Consider the transformation 
\begin{equation}
\label{problem1}
v_{n+1} = (a+1) v_n - av_{n-1}, \qquad n \in \mathbb{N}_0,
\end{equation}
of equation \eqref{problem} obtained through the change of variable $v_{n+1}/ v_n := a + x_{n+1} x_{n-k+1}$.
We determine the solution form of equation \eqref{problem1} through a classical method in solving linear (homogenous) recurrences. 
That is, we use a discrete function $\lambda^n$ where $\lambda \in \mathbb{C}\setminus \{0\}$ and $n \in \mathbb{N}_0\cup \{-1\}$ to a obtain a Binet-like form of the $n$-th term $v_n$.
The inclusion of the index $-1$ is crucial (as we are using the ansatz $v_n = \lambda^n$) in this approach, 
and this we shall see as we proceed in our discussion.
It is worth noting that there are many other techniques in solving linear recurrences with constant coefficients (see, e.g., \cite{ams} and \cite{larcombe}), 
and here we shall apply the method of using a discrete function as our main approach.
However, we shall also remark that the method of differences, much known as \emph{telescoping sums}, 
can be effectively used to derive the solution form of equation \eqref{problem}.
For an interesting application of this method to a class of difference equations, we refer the readers to \cite{rabago4}.

Now, to begin the computation, we let $v_n = \lambda^n$ for some $\lambda \in \mathbb{C} \setminus \{0\}$ and $n \in \mathbb{N}_0 \cup \{-1\}$.
From equation \eqref{problem1}, we get
$
\lambda^{n+1} = (a+1)\lambda^n - a \lambda^{n-1}
$
or equivalently,
$
\lambda^2 - (a+1)\lambda + a = 0
$
whose roots are given by $\lambda_1 = a$  and $\lambda_2 = 1$.
Since $a \neq 1$, then it is evident that $\lambda_1$ and $\lambda_2$ are distinct. 
Therefore, by a standard result in difference equations, we can write $v_n$ as 
$
v_n = c_1 a^n + c_2 1^n
$
for some computable constants $c_1$ and $c_2$.
These coefficients are easily determined by computing the solution pair $(c_1,c_2)$  of the system
\[
\begin{cases}
c_1 + c_2 = v_0\\
c_1 + ac_2 = av_{-1}.
\end{cases}
\]
Thus, 
\[
c_1 = \frac{av_0 - av_{-1}}{a-1}\qquad \text{and}\qquad c_2 = -\frac{v_0 - av_{-1}}{a-1},
\]
from which it follows that
\[
v_n = \left( \frac{a^{n+1} - 1}{a-1} \right)v_0 - a\left( \frac{a^n -1}{a-1}\right) v_{-1}.
\]
Form the relation $x_{n} x_{n-k} = v_n/v_{n-1} - a$, we obtain
\begin{align}
	x_{n} x_{n-k} 
	&=  \frac{ \left( \frac{a^{n+1} - 1}{a-1} \right)\frac{v_0}{v_{-1}} - a\left( \frac{a^n -1}{a-1}\right) }{ \left( \frac{a^{n} - 1}{a-1} \right)\frac{v_0}{v_{-1}} - a\left( \frac{a^{n-1} -1}{a-1}\right)} -a\nonumber\\
	&= \frac{ \left( \frac{a^{n+1} - 1}{a-1} \right)(a+x_0 x_{-k}) - a\left( \frac{a^n -1}{a-1}\right)}{ \left( \frac{a^{n} - 1}{a-1} \right)(a+x_0 x_{-k}) - a\left( \frac{a^{n-1} -1}{a-1}\right)} -a \nonumber\\
	&= \frac{ x_0 x_{-k} }{ a^{n} + \left(\frac{a^{n} - 1}{a-1} \right)x_0 x_{-k}}.\label{xxform}
\end{align}
Now, replacing $n$ by $2kj+i$ (resp., $2kj-k+i$) for $i \in I:= \{ -k, -k+1, \ldots, 0\}$, we get
\begin{align*}
	x_{2kj+i} x_{2kj-k+i} &= \frac{ x_0 x_{-k} }{ a^{2kj+i} + \left(\frac{a^{2kj+i} - 1}{a-1} \right)x_0 x_{-k}},\\
	x_{2kj-k+i} x_{2kj-2k+i} &= \frac{ x_0 x_{-k} }{ a^{2kj-k+i} + \left(\frac{a^{2kj-k+i} - 1}{a-1} \right)x_0 x_{-k}},
\end{align*}
respectively.
Taking the ratio of the corresponding sides of the above equations, and then taking the product of the resulting expression from $j=1$ to $j=n$, we get
\begin{align}
	x_{2kn+i} &= x_{i}\prod_{j=1}^n \left\{ \frac{x_{2kj+i} x_{2kj-k+i} }{x_{2kj-k+i} x_{2kj-2k+i} }\right\} \nonumber\\
	&= x_i \prod_{j=1}^n \left\{ \frac{ a^{2kj-k+i} + \left(\frac{a^{2kj-k+i} - 1}{a-1} \right)x_0 x_{-k} }{ a^{2kj+i} + \left(\frac{a^{2kj+i} - 1}{a-1} \right)x_0 x_{-k}} \right\},
	\label{exform}
\end{align}
for all $n\in \mathbb{N}_0$ and $i \in I$, with the usual convention that $\prod_{j=1}^0 (\cdot) = 1$.
Notice that, with the above indices of the iterate, we were not able to describe the form of the first $k-1$ iterates $\{x_n\}_{n=1}^{k-1}$.
However, these iterates can be obtained easily by replacing $i$ by $-i$ in \eqref{exform} and let $i$ run from $1$ to $k-1$.
Alternatively, we can utilize equation \eqref{xxform} and let $n$ assumes the value from $1$ to $k-1$.
More precisely, we have
\[
	x_n = \frac{ x_0 x_{-k} }{x_{n-k}} \left( a^{n} + \left(\frac{a^{n} - 1}{a-1} \right)x_0 x_{-k}\right)^{-1},
\] 
for all $n = 1, 2, \ldots, k-1$.
\begin{rem}
\label{firstremark}
We note that we can determine the solution form of equation \eqref{problem1} via telescoping sums.
To do this, we transform equation \eqref{problem1} to the equivalent form $v_{n+1} - v_n =a(v_n - v_{n-1})$.
Letting $w_{n+1}: = v_{n+1} - v_n$, we can write equation \eqref{problem1} as $w_{n+1} = a w_{n}$ which, upon iterating the right-hand side, leads to $w_{n+1} = a^{n+1} w_0$.
This equation, in turn, yields the relation $v_{n+1} - v_n = a^{n+1} (v_0 - v_{-1})$,
and by telescoping sums we easily obtain the identity
\begin{align*}
v_n = v_{-1} +  (v_0 - v_{-1})\sum_{j=-1}^{n-1} a^{j+1}
	= \left( \frac{a^{n+1} - 1}{a-1} \right)v_0 - a\left( \frac{a^n -1}{a-1}\right) v_{-1}.
\end{align*}
To this end, one can follow the same inductive lines as above to get the desired result.
Referring to the form of $v_n$ computed above, it is clear that $a$ must not equate to unity since, if it is so, the quantity will be undefined.
\end{rem}
In concluding, we have just proved the following result.
\begin{thm}
\label{maintheorem}
	Let $k\in \mathbb{N}$ and $a>1$ ($a\neq1$) be fixed.
	Then, every well-defined solution $\{x_n\}_{n=-k}^{\infty}$ of equation \eqref{problem} takes the form
	\begin{equation}
	x_n = \frac{ x_0 x_{-k} }{x_{n-k}} \left[ a^{n} + \left(\frac{a^{n} - 1}{a-1} \right)x_0 x_{-k}\right]^{-1},
	\label{thmform1}
	\end{equation}
	for all $n = 1, 2, \ldots, k-1$, and 
	\begin{equation}
	x_{2kn+i} = x_i \prod_{j=1}^n \left\{ \frac{ a^{2kj-k+i} + \left(\frac{a^{2kj-k+i} - 1}{a-1} \right)x_0 x_{-k} }{ a^{2kj+i} + \left(\frac{a^{2kj+i} - 1}{a-1} \right)x_0 x_{-k}} \right\},
	\label{thmform2}
	\end{equation}
	for all $n\in \mathbb{N}_0$ and $i \in I$.
	If, in addition, $x_0 x_{-k} = 1 -a$, $a\neq 1$, then the solution forms \eqref{thmform1} and \eqref{thmform2} can be simplified as
	\[
	x_n = \frac{ x_0 x_{-k} }{x_{n-k}} ,	
	\]
	for all $n = 1, 2, \ldots, k-1$, and 
	\[
	x_{2kn+i} = x_i,
	\]
	for all $n\in \mathbb{N}_0$ and $i \in I$, respectively.

	By a well-defined solution of \eqref{problem}, we mean a solution sequence $\{x_n\}_{n=-k}^{\infty}$ with real initial conditions $\{x_n\}_{n=-k}^0$ 
	such that $x_i \neq 0$ for all $i \in I$, and
	\[
	x_0 x_{-k} \not\in A\cup B:=
	\left\{- a^j \left(\frac{a^j - 1}{a-1} \right)^{-1} \right\}_{j=1}^{k-1} \bigcup \left\{- a^{2kj+i} \left(\frac{a^{2kj+i} - 1}{a-1} \right)^{-1} \right\}_{j=1}^{\infty},
	\] 
	for all $i \in I$.
\end{thm}
As an immediate consequence of Theorem \ref{maintheorem}, we finally obtain the forbidden set for the difference equation \eqref{problem} given in the following corollary.
\begin{cor}
\label{cor1}
Let $\boldsymbol{x}_0 := (x_{-k}, x_{k+1}, \ldots, x_0) \in \mathbb{R}^{k+1}$, $k \in \mathbb{N}$, and $a>1$ ($a\neq 1$) be fixed. 
Then, the forbidden set $\mathcal{F}$ of the difference equation \eqref{problem} is given by
\[
\mathcal{F} = \left\{\boldsymbol{x}_0\ : \ x_i \neq 0 \ \text{for all $i \in I$, and}\   x_0 x_{-k} \not\in A \cup B \right\}.
\]
\end{cor}
\begin{rem}
	We observe that the computation of the closed form solution of equation \eqref{problem} does not require the positivity of $a$.
	This suggests that, following the same line of arguments, the results can easily be extended to the case when $a$ is an arbitrary real number not equal to zero,
	or possibly when $a \in \mathbb{C}\setminus \{0\}$ in general. 
\end{rem}


\section{Some Results on the Behavior of Solutions of Equation \eqref{problem}}

In this section we examine the case when $a$ is the unity, and present some results regarding the qualitative behavior of the solution of equation \eqref{problem}.
Also, we provide some numerical illustrations depicting the long-time behavior of solutions of equation \eqref{problem} for some given fixed constants $k \in \mathbb{N}$ and $a>0$.

Before we proceed further, we need to recall what we mean by an \emph{eventually periodic solution}.

\begin{defn}[{\cite{ladas}}]
\label{periodicity}
Let $k \in \mathbb{N}$. A sequence $\{x_n\}_{n=-k}$ is said to be periodic with period $p$ if $x_{n+p} =x_n$, for all $n\geq-k$.
Moreover, a solution $\{x_n\}_{n=-k}$ of \eqref{problem} is called eventually periodic with period $p$ if there exists an integer $N \geq -k$ such that  $\{x_n\}_{n=-k}$ is periodic with period $p$; that is, $x_{n+p} = x_n$, for all $n \geq N$.
\end{defn}

Hereinafter, we assume that $\{x_n\}_{n=-k}^{\infty}$ is a well-defined solution of \eqref{problem}.
 
\subsection{Form and Periodicity of Solutions for the Case $a=1$}

For the case when $a = 1$, we have the following corollary of Theorem \ref{maintheorem}.
\begin{cor}
\label{cor2}
If $a =1$, then every solution $\{x_n\}_{n=-k}^{\infty}$ of equation \eqref{problem} takes the form
	\begin{equation}
	x_n = \frac{ x_0 x_{-k} }{x_{n-k}} \left(1 + nx_0 x_{-k}\right)^{-1},
	\label{corform1}
	\end{equation}
	for all $n = 1, 2, \ldots, k-1$, and 
	\begin{equation}
	x_{2kn+i} = x_i \prod_{j=1}^n \left\{ \frac{ 1 + \left(2kj-k+i\right)x_0 x_{-k} }{ 1 + \left({2kj+i} \right)x_0 x_{-k}} \right\},
	\label{corform2}
	\end{equation}
	for all $n\in \mathbb{N}_0$ and $i \in I$.
	Furthermore, the solution sequence $\{x_n\}_{n=-k}^{\infty}$ is eventually periodic with period $2k$. 
	In addition, the forbidden set $\mathcal{F}_1$ of equation \eqref{problem} for $a=1$ is given by
	\[
	\mathcal{F}_1 = 
	\left\{\boldsymbol{x}_0\ : \ x_i \neq 0 \ \text{for all $i \in I$, and}\   x_0 x_{-k} \not\in \left\{ -\frac1j\right\}_{j=1}^{k-1} \bigcup \left\{ -\frac{1}{2kj+1}\right\}_{j=1}^{\infty} \right\},
	\]
	where $\boldsymbol{x}_0 := (x_{-k}, x_{k+1}, \ldots, x_0) \in \mathbb{R}^{k+1}$.
\end{cor} 

\begin{proof}
Let $\{x_n\}_{n=-k}^{\infty}$ be a solution of \eqref{problem} with $a=1$, and $\mathcal{F}_1$ denotes its forbidden set.
Formulas \eqref{corform1} and \eqref{corform2} follow directly by taking the limit of equations \eqref{thmform1} and \eqref{thmform2}, respectively, 
as $a$ approaches the unity. 
Meanwhile the periodicity of $\{x_n\}_{n=-k}^{\infty}$ is a consequence of the fact that 
\[
 \frac{ 1 + \left(2kj-k+i\right)x_0 x_{-k} }{ 1 + \left({2kj+i} \right)x_0 x_{-k}} \longrightarrow 1
 \]
 as $j \to \infty$, and of course, as long as $\boldsymbol{x}_0 \not\in \mathcal{F}_1$.
 Finally, the forbidden set $\mathcal{F}_1$ is specified by finding the values of $x_0x_{-k}$ for which formulas \eqref{corform1} and \eqref{corform2} are undefined. 
\end{proof}


\subsection{Limiting Properties of Solutions for $a\in \mathbb{R}^+\setminus\{1\}$}
\label{sec3pt2}

Now, we examine the limiting properties of solutions of equation \eqref{problem} for $a>0$ not equal to the unity.
First, we investigate the possibility that a solution to \eqref{problem} is convergent to zero.
To see this possibility, it suffices to determine when is the subsequence $\{x_{2kn+i}\}_{n=1}^{\infty}$, for all $i \in I$, converges to zero.
This situation would only be possible when $|x_{2kn+i}| < |x_{2k(n-1)+i}|$, for all $n \geq 2$ and $i \in I$.
In view of equation \eqref{thmform2}, this condition is equivalent to
\[
(0<)\ \left| \frac{ a^{2kn-k+i} + \left(\frac{a^{2kn-k+i} - 1}{a-1} \right)x_0 x_{-k} }{ a^{2kn+i} + \left(\frac{a^{2kn+i} - 1}{a-1} \right)x_0 x_{-k}} \right| < 1,
\]
for all $n\in \mathbb{N}$ and $i \in I$.
Without-loss-of-generality, suppose that the numerator and the denominator are both positive.
Then, after some rearrangement, the above inequality condition can be expressed as
\[
	 0 
	 < a^{2kn-k+i}(a^k - 1)\left(  1 + \frac{x_0x_{-k}}{a-1}\right).
\]
Since this inequality must hold true for all $n\in \mathbb{N}$ and $i \in I$,
then $a$ must be greater than the unity and $x_0x_{-k} \neq 1-a$.
Given these conditions, we conclude that every solution $\{x_n\}_{n=-k}^{\infty}$ of equation \eqref{problem} will converge to zero for $a > 1$.

Similarly, we can show, without any difficulty, that every solution $\{x_n\}_{n=-k}^{\infty}$ of equation \eqref{problem} 
will eventually be periodic whenever $a \in (0, 1)$ or $x_0 x_{-k} = 1 - a$.
In either of these situations, the periodicity is given by $2k$.
Indeed, for $a \in (0,1)$, the quantity $a^{2kj-k+i}$ vanishes as $j$ goes to infinity, for all $i \in I$.
In addition, the ratio $(a^{2kj+i} - 1)/(a-1)$ will converge to $1/(a-1)$ as $j$ goes to infinity, for all $i \in I$.
These results imply that
\[
\frac{ a^{2kj-k+i} + \left(\frac{a^{2kj-k+i} - 1}{a-1} \right)x_0 x_{-k} }{ a^{2kj+i} + \left(\frac{a^{2kj+i} - 1}{a-1} \right)x_0 x_{-k}} \longrightarrow 1,
\]
as $j \to \infty$ (and of course, given that $\boldsymbol{x}_0 \not\in \mathcal{F}$).
Meanwhile, the case when $x_0 x_{-k} = 1 - a$ ($a\neq 1$), has already been discussed in Section \ref{sec2}, and we will not repeat it here.
In summary, we see that following result holds.
 
 \begin{thm}
 \label{limits}
 	Let $\{x_n\}_{n=-k}^{\infty}$ be a solution of equation \eqref{problem}.
	If $a>1$ and $x_0 x_{-k} \neq 1-a$, then $\{x_n\}_{n=-k}^{\infty}$ converges to zero.
	If, however, $a \in (0,1)$ or $x_0 x_{-k} = 1-a$, then $\{x_n\}_{n=-k}^{\infty}$ is eventually periodic with period $2k$. 
 \end{thm}
 
 \subsection{Numerical Examples}
 
 Finally, in this section, we provide some numerical examples that illustrate our results in the previous two sections.
 In these examples (see Figure \ref{figures}), the initial conditions $\{x_n\}_{n=-k}^0$ are chosen randomly on the interval $(-1,1)$.
 The results verify Corollary \ref{cor2} and Theorem \ref{limits}.
   
  \begin{figure}[h!]
  \centering
  	\scalebox{.5}{\includegraphics{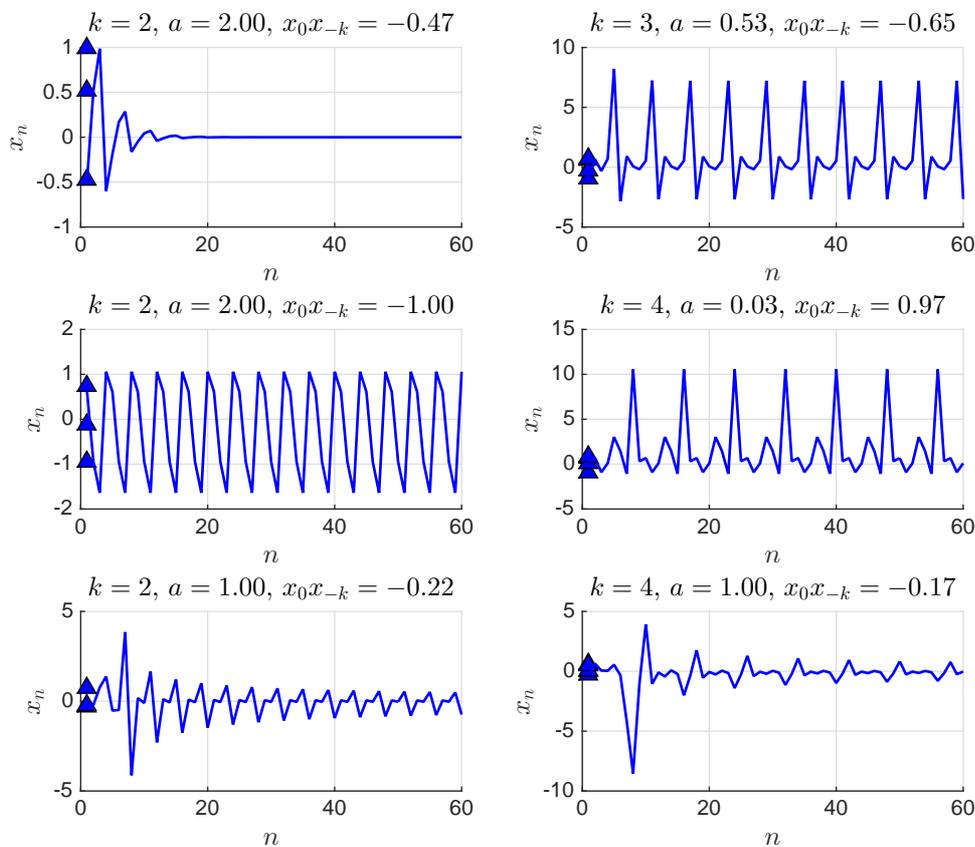}}
	\caption{The uppermost plots corroborate the results in Theorem \ref{limits} for the case $a>1$ (left plot) and the case $a\in(0,1)$ (right plot).
	Meanwhile, the middle plots illustrate the case when $a > 1$ and $x_0x_{-k} \neq 1$ (right plot), 
	and $a\in(0,1)$ (left plot) given that $x_0x_{-k} = 1-a$.
	Clearly, the solution sequences are both periodic with period four and eight, respectively.
	Finally, the last two (lower) plots illustrate two particular situations when $a=1$.
	Evidently, the figures show that, in these situations, the solutions to \eqref{problem} are eventually periodic.
	The solution sequences (left and right) have periods four and eight, respectively.   
	This verify the results stated in Corollary \ref{cor2}.
	}\label{figures}
  \end{figure}

\section{Summary and a Possible Extension}
We have successfully settled one of the open problem raised by Balibrea and Cascales in \cite{bc}.
The solution form to the given rational difference equation was established by reducing the equation to a linear type difference equation.
The resulting equation was then solved through a classical method in solving linear homogenous recurrence equation with constant coefficients.
We emphasize that the method used here can obviously be applied to other problems offered in \cite{bc}, 
especially to those nonlinear difference equations whose solution form are, in structure, similar to the ones obtained here.
In fact, we believe that the method employed here can be used effectively in examining the case when $a$ is replaced by a $2k$-periodic sequence of real or complex numbers. 
Consequently, we believe that the discussion delivered here provides a better understanding of the forbidden set problem in the frame of rational difference equations, and had provided considerable interest in examining other classes of nonlinear difference equations.

As a possible generalization of the open problem addressed in this work, we mention that the case when $a$ is replaced by a general number sequence is also an interesting problem to investigate.
So, we ask, given fixed constants $k\in \mathbb{N}$ and $a \in \mathbb{C}\setminus \{0\}$, and a general number sequence $\{a_n\}_{n=0}^{\infty}$, 
what is the corresponding forbidden set for the rational difference equation
\[
x_{n+1} = \dfrac{x_n x_{n-k}}{a_n x_{n-k+1} +x_n x_{n-k+1} x_{n-k}},
\]
with real (or complex) initial conditions $\{x_n\}_{n=-k}^0$.
Finally, we announce that the other open problems presented in \cite{bc} shall be the subject of our future investigations elsewhere.


\section*{Acknowledgment}

The author would like to express his gratitude to the anonymous referees whose valuable comments and suggestions helped him to improve the quality of the paper.

\end{document}